\documentclass[11pt]{amsart}

\usepackage{amsfonts, amssymb, amsthm, amsmath, gensymb, dsfont}
\usepackage{graphicx, comment}
\usepackage[small]{caption}
\usepackage{subcaption}
\usepackage{epsfig}
\usepackage{tikz, pgfplots, float}

\usepackage[utf8]{inputenc}
\usepackage{fullpage}
\usepackage{framed}

\usepackage{enumerate}
\usepackage{url}
\usepackage[breaklinks]{hyperref}
\usepackage{cleveref}
\hypersetup{
	colorlinks = true, 
	urlcolor = cyan, 
	linkcolor = teal, 
	citecolor = cyan 
}

\newtheorem{theorem}{Theorem}
\newtheorem{definition}[theorem]{Definition}
\newtheorem{lemma}[theorem]{Lemma}

\crefname{claim}{claim}{claims}

\newtheorem{proposition}[theorem]{Proposition}
\newtheorem{corollary}[theorem]{Corollary}

\newcommand{\eps}{\varepsilon}

\newcommand{\ex}{\mathrm{ex}}
\newcommand{\cS}{\mathcal{S}}
\newcommand{\FF}{\boldsymbol{F}}
\newcommand{\cD}{\boldsymbol{\mathcal D}}
\newcommand{\cl}{c\ell}

\usepackage{tikz}
\tikzstyle{vtx} = [circle, fill, inner sep=0.7]

\title{Extremal problems for suspensions of even cycles}

\author{Dingyuan Liu}
\address{Dingyuan Liu \newline Karlsruhe Institute of Technology, Germany}
\email{liu@mathe.berlin}

\date{}

\begin{document}

\begin{abstract}
Given an integer $k\geq2$ and a graph $F$, the $k$-uniform suspension $\cS^kF$ is obtained by adjoining a fixed set of $k-2$ new vertices to every edge of $F$. In this paper, we study two extremal problems for suspensions of even cycles.

Write $K^k_t$ for the $k$-uniform clique of order $t$. Let $\ex(n,\cS^kC_{2\ell})$ and $\ex(n,K^k_t,\cS^kC_{2\ell})$ denote the maximum numbers of edges and copies of $K^k_t$, respectively, in an $\cS^kC_{2\ell}$-free $k$-uniform hypergraph on $n$ vertices. We prove that, for every $k\geq2$ and infinitely many $n$,
\[\ex(n,K^{k}_{k+1},\cS^kC_4)=\frac{n^{k-1/2}}{(k+1)!}+O(n^{k-1}).\]
This extends a folklore result for $k=2$ and, as an immediate consequence, yields the asymptotics of $\ex(n,\cS^kC_4)$ for infinitely many $n$, previously established by Mubayi (for all $n$).

Furthermore, for every $k\geq2$ and $\ell\in\{3,5\}$, we determine the order of magnitude
\[\ex(n,\cS^kC_{2\ell})=\Theta(n^{k-1+1/\ell}).\]
This generalizes both the classical graph case $k=2$ and a previous result of Mukherjee for $k=\ell=3$.

The principal difficulty in both problems lies in constructing the lower bounds. Our construction for $\ex(n,K^{k}_{k+1},\cS^kC_4)$ incorporates a novel block-packing structure, which yields substantially more copies of $K^k_{k+1}$ than previously known constructions. For $\ex(n,\cS^kC_{2\ell})$ with $\ell\in\{3,5\}$, we establish a natural $k$-uniform version of Wenger graphs, addressing the subtleties involved in lifting extremal graph constructions to suspensions. We also give applications of our results to Tur\'an problems for simplicial complexes.
\end{abstract}

\maketitle

\section{Introduction}
\label{sec:intro}
Given a graph $F$, the Tur\'an number $\ex(n,F)$ is the maximum number of edges in an $n$-vertex $F$-free graph. Building on earlier work of Tur\'an~\cite{T41}, the classical Erd\H{o}s--Stone--Simonovits theorem~\cite{ES66,ES46} states that, for any fixed graph $F$ containing at least one edge,
\[\ex(n,F)=\left(1-\frac{1}{\chi(F)-1}+o(1)\right)\binom{n}{2},\]
where $\chi(F)$ denotes the chromatic number of $F$. The theorem asymptotically determines the Tur\'an number of every non-bipartite graph, leaving the bipartite case as a major open problem in extremal graph theory. As even cycles are arguably among the most natural bipartite graphs, their Tur\'an numbers have been extensively studied. Bondy and Simonovits~\cite{BS74} proved that, for any fixed $\ell\geq2$,
\begin{equation}
\label{C6C10}
\ex(n,C_{2\ell})=O(n^{1+1/\ell}).
\end{equation}
On the other hand, Lazebnik, Ustimenko, and Woldar~\cite{LUW95} obtained the best known general lower bound $\ex(n,C_{2\ell})=\Omega(n^{1+1/\lceil{(3\ell-3)/2}\rceil})$. Lower bounds matching~\eqref{C6C10} are known only for $\ell\in\{2,3,5\}$. The case $\ell\in\{3,5\}$ dates back to constructions of Benson and Singleton~\cite{Be66,S66}; see also Wenger~\cite{W91}. For the case $\ell=2$, one has a sharper estimate
\begin{equation}
\label{C4}
\frac{n^{3/2}}{2}-o(n^{3/2})\leq\ex(n,C_4)\leq\frac{n^{3/2}}{2}+O(n).
\end{equation}
The upper bound is due to Reiman~\cite{R58}, while matching constructions were obtained independently by Brown~\cite{Br66} and by Erd\H{o}s, R\'enyi, and S\'os~\cite{ERS66}; see also F\"uredi~\cite{F83}.

The Tur\'an problem for graphs extends readily to the hypergraph setting. For an integer $k\geq2$ and a $k$-uniform hypergraph $H$, let $\ex(n,H)$ denote the maximum number of edges in an $n$-vertex $H$-free $k$-uniform hypergraph. As one would expect, our understanding of $\ex(n,H)$ is considerably more limited than in the graph setting; see, e.g., the survey by Keevash~\cite{K11}. This motivates the study of hypergraphs whose extremal behavior might be governed by an underlying graph structure.

The suspension operation provides a natural way to lift a graph to a uniform hypergraph. Given an integer $k\geq2$ and a graph $F$, the $k$-uniform suspension $\cS^kF$ is obtained by adjoining a fixed set of $k-2$ new vertices to every edge of $F$. That is, $\cS^kF$ is a $k$-uniform hypergraph with vertex set $S\cup V(F)$ and edge set $\{S\cup e:\,e\in E(F)\}$, where $S$ is a $(k-2)$-element set disjoint from $V(F)$. The extremal numbers of suspension hypergraphs have long been studied in various formulations; see the discussion in~\cite{CGKZ26} for an overview. In particular, since $\cS^kF$ is $k$-partite if and only if $F$ is bipartite, by a classical result of Erd\H{o}s~\cite{E64}, we have
\[\ex(n,\cS^kF)=\begin{cases}
    \Theta(n^k) & \text{ if $\chi(F)\geq3$,}\\
    O(n^{k-\eps}) & \text{ if $\chi(F)\leq2$.}
\end{cases}\]
where $\eps>0$ is some constant depending on $F$. Thus, the order of magnitude of $\ex(n,\cS^kF)$ again remains undetermined only when $F$ is bipartite. Mubayi~\cite{M02} considered complete bipartite graphs $F=K_{s,t}$ and determined the order of magnitude for certain ranges of $s$ and $t$. In the special case $F=K_{2,2}=C_4$, he obtained a sharper estimate
\begin{equation}
\label{Mubayi_C4}
\frac{n^{k-1/2}}{k!}+o(n^{k-1/2})\leq\ex(n,\cS^kC_4)\leq\frac{n^{k-1/2}}{k!}+O(n^{k-1})
\end{equation}
for every $k\geq2$. The upper bound in~\eqref{Mubayi_C4} follows by combining~\eqref{C4} with a simple counting argument, while the lower bound follows by generalizing the construction of $C_4$-free graphs by F\"uredi~\cite{F83} to the $k$-uniform setting.

For $k$-uniform hypergraphs $T$ and $H$, the generalized Tur\'an number $\ex(n,T,H)$ is defined as the maximum number of copies of $T$ in an $n$-vertex $H$-free $k$-uniform hypergraph. Note that when $T$ is a single edge, we have $\ex(n,T,H)=\ex(n,H)$. The systematic study of $\ex(n,T,H)$ was initiated by Alon and Shikhelman~\cite{AS16} and has since been pursued extensively; see the survey by Gerbner and Palmer~\cite{GP25} for related results. One of the earliest results on generalized Tur\'an problems concerns the case $T=K_3$ and $H=C_4$. It is folklore that
\begin{equation}
\label{cliques_C4}
\frac{n^{3/2}}{6}-o(n^{3/2})\leq\ex(n,K_3,C_4)\leq\frac{n^{3/2}}{6}+O(n).
\end{equation}
The upper bound follows from the corresponding bound on $\ex(n,C_4)$ together with the observation that every edge of a $C_4$-free graph is contained in at most one triangle. The lower bound can be obtained from the same construction used for $\ex(n,C_4)$~\cite{Br66,ERS66,F83} or via the Erd\H{o}s--R\'enyi graph~\cite{ER62}. See also the proof by Alon and Shikhelman~\cite{AS16} of a more general result.

A natural hypergraph extension of~\eqref{cliques_C4} is to count cliques in an $\cS^kC_4$-free $k$-uniform hypergraph. Note that any such hypergraph contains no clique of order greater than $k+1$. Moreover, the case of cliques of order $k$ corresponds to precisely $\ex(n,\cS^kC_4)$, which was determined asymptotically by Mubayi~\cite{M02}. Therefore, the only remaining case of interest is $\ex(n,K^k_{k+1},\cS^kC_4)$. An upper bound analogous to that for $\ex(n,K_3,C_4)$ follows readily. Indeed, every edge of an $\cS^kC_4$-free $k$-uniform hypergraph is contained in at most one clique of order $k+1$, since two such cliques sharing an edge would give rise to a copy of $\cS^kC_4$. Together with~\eqref{Mubayi_C4}, this yields
\begin{equation}
\label{upper_bound}
\ex(n,K^k_{k+1},\cS^kC_4)\leq\frac{\ex(n,\cS^kC_4)}{k+1}=\frac{n^{k-1/2}}{(k+1)!}+O(n^{k-1}).
\end{equation}

On the other hand, if an $\cS^kC_4$-free $k$-uniform hypergraph contains $\frac{n^{k-1/2}}{(k+1)!}+O(n^{k-1})$ copies of $K^k_{k+1}$, then it necessarily has $\frac{n^{k-1/2}}{k!}+O(n^{k-1})$ edges, with almost every edge contained in one of the copies. This suggests that obtaining a lower bound matching~\eqref{upper_bound} is considerably more delicate than determining $\ex(n,\cS^kC_4)$. Indeed, Mubayi's construction~\cite{M02} for $\ex(n,\cS^kC_4)$ contains roughly $\binom{\sqrt{n}}{k+1}$ copies of $K^k_{k+1}$, which is of lower order than $n^{k-1/2}$ when $k\geq3$.

Here, we develop a new algebraic construction of an $\cS^kC_4$-free $k$-uniform hypergraph containing substantially more copies of $K^k_{k+1}$. The novelty of our construction lies in a block-packing structure, with each block corresponding to a copy of $K^k_{k+1}$, while the resulting hypergraph remains $\cS^kC_4$-free. Using this construction, we establish a tight lower bound on $\ex(n,K^k_{k+1},\cS^kC_4)$ for every $k\geq2$ and infinitely many values of $n$.

\begin{theorem}
\label{thm:C4}
Let $k\geq2$ be an integer and let $n=2^m$ be a power of two, where $m\geq2$. Then
\[\ex(n,K^k_{k+1},\cS^kC_4)\geq\frac{n^{k-1/2}}{(k+1)!}-n^{k-1}\]
if $m$ is even, and
\[\ex(n,K^k_{k+1},\cS^kC_4)\geq\frac{n^{k-1/2}}{\sqrt{2}(k+1)!}-n^{k-1}\]
if $m$ is odd.
\end{theorem}

In particular,~\Cref{thm:C4} and~\eqref{upper_bound} yield that
\[\ex(n,K^{k}_{k+1},\cS^kC_4)=\frac{n^{k-1/2}}{(k+1)!}+O(n^{k-1})\]
when $n$ ranges over the sequence of powers of four. For general $n$, we obtain the following bounds on $\ex(n,K^k_{k+1},\cS^kC_4)$ by combining~\Cref{thm:C4} with a random-sampling argument.

\begin{corollary}
\label{cor:C4}
Let $k\geq2$ be an integer. Then
\[\frac{n^{k-1/2}}{2^{2-3/(k+1)}(k+1)!}-O(n^{k-1})\leq\ex(n,K^k_{k+1},\cS^kC_4)\leq\frac{n^{k-1/2}}{(k+1)!}+O(n^{k-1}).\]
\end{corollary}

We remark that the constant $2^{2-3/(k+1)}$ appearing in the denominator of the leading term is not optimal even within our proof. However, since we suspect that the optimal constant is $1$ for all $n$, we make no attempt here to optimize it further.

Unfortunately, the construction for~\Cref{thm:C4} does not extend to suspensions of longer cycles in a straightforward way. As discussed earlier, estimating $\ex(n,K^k_{k+1},\cS^kC_{2\ell})$ is considerably more difficult than estimating $\ex(n,\cS^kC_{2\ell})$. Since the order of magnitude of $\ex(n,\cS^kC_{2\ell})$ remains unknown in general, an estimate of the correct order for $\ex(n,K^k_{k+1},\cS^kC_{2\ell})$ seems out of reach. Nevertheless, using a probabilistic construction, we can obtain the following lower bound.

\begin{proposition}
\label{prop:longer_cycles}
Let $k\geq2$ and $\ell\geq3$ be integers. For every integer $t$ with $k\leq t<k+\ell-2$,
\[\ex(n,K^k_t,\cS^kC_\ell)=\Omega(n^{k-1+1/(\ell-1)})\]
and
\[\ex(n,K^k_t,\cS^kC_\ell)=\begin{cases}
O(n^k) & \text{if $\ell$ is odd,}\\
O(n^{k-1+2/\ell}) & \text{if $\ell$ is even.}
\end{cases}\]
\end{proposition}

A special case of~\Cref{prop:longer_cycles} yields
\[\Omega(n^{k-1+1/(2\ell-1)})\leq\ex(n,\cS^kC_{2\ell})\leq O(n^{k-1+1/\ell}),\]
which are the best known general bounds for all fixed $k\geq3$ and $\ell\notin\{2,3,5\}$.

Recall that, as with $C_4$, the Tur\'an numbers of $C_6$ and $C_{10}$ are determined up to a constant factor. One would then expect to determine the correct order of magnitude of $\ex(n,\cS^kC_{2\ell})$ for $\ell\in\{3,5\}$ as well. Nevertheless, when $k\neq2$, the only progress is due to Mukherjee~\cite{M24}, who proved that $\ex_3(n,\cS^3C_6)=\Theta(n^{7/3})$. The obstacle lies in establishing the lower bounds. In contrast to the case of $C_4$, where the construction from~\cite{F83} can be extended in a relatively straightforward way, lifting the constructions for $C_6$ and $C_{10}$~\cite{Be66,S66,W91} to higher uniformity appears rather non-trivial. This difficulty presumably led Mukherjee~\cite{M24} to develop a different construction for $\cS^3C_6$.

Our second main result breaks through the barrier and, by establishing $k$-uniform analogues of Wenger's constructions~\cite{W91}, determines the order of magnitude of $\ex(n,\cS^kC_{2\ell})$ for every $k\geq2$ and $\ell\in\{3,5\}$.

\begin{theorem}
\label{thm:C6C10}
Let $k\geq2$ be an integer and $\ell\in\{3,5\}$. Then
\[\ex(n,\cS^kC_{2\ell})=\Theta(n^{k-1+1/\ell}).\]
\end{theorem}

The paper is structured as follows. In~\Cref{sec:C4} we prove~\Cref{thm:C4} and~\Cref{cor:C4}. In~\Cref{sec:C_ell} we give the short proof of~\Cref{prop:longer_cycles}. In~\Cref{sec:C6C10} we establish~\Cref{thm:C6C10}. Finally, in~\Cref{sec:applications} we present applications of our results to Tur\'an problems for simplicial complexes.

\section{Counting cliques in $\cS^kC_4$-free hypergraphs}
\label{sec:C4}
This section is devoted to proving~\Cref{thm:C4} and~\Cref{cor:C4}. We will first present in~\Cref{sec:construction} the block-packing construction, which serves as the cornerstone of our proofs, and verify that this construction is $\cS^kC_4$-free. In~\Cref{sec:counting} we show that this construction contains sufficiently many blocks, each corresponding to a copy of $K^k_{k+1}$. In~\Cref{sec:proof_thm} we prove~\Cref{thm:C4}, and in~\Cref{sec:proof_cor} we deduce~\Cref{cor:C4} from~\Cref{thm:C4} via a random-sampling argument.

\subsection{The block-packing construction}
\label{sec:construction}
Recall that, a matching lower bound on $\ex(n,K^k_{k+1},\cS^kC_4)$ requires an $\cS^kC_4$-free hypergraph where almost every edge is contained in exactly one copy of $K^k_{k+1}$. This naturally suggests building such hypergraphs not edge by edge, but rather by packing copies of $K^k_{k+1}$. Our goal is to pack as many copies of $K^k_{k+1}$ as possible while keeping their union $\cS^kC_4$-free. We achieve this by selecting the copies subject to tailored algebraic constraints.

Let $2\leq p\leq q$ be powers of two. Let $A\leq(\mathbb{F}_q,+)$ be an additive subgroup of order $p$. Let
\[\Gamma=A\times\mathbb{F}_q\subseteq\mathbb{F}_q^2\]
be an additive group. Since $p$ and $q$ are both powers of two, the group $\Gamma$ has exponent two, that is, $x+x=0$ for all $x\in\Gamma$. This property is essential to our construction and leads to the restriction on the values of $n$ in~\Cref{thm:C4}.

For $x=(x_1,x_2)$ and $y=(y_1,y_2)$ in $\mathbb{F}_q^2$, define the symmetric bilinear form
\[\langle{x,y}\rangle=x_1y_2+x_2y_1.\]
Since $\mathbb{F}_q$ has characteristic two, $\langle{x,x}\rangle=0$ holds for all $x\in\mathbb{F}_q^2$. For a subset $X\subseteq\Gamma$, let
\[\sigma(X)=\sum_{x\in X}x\qquad\text{and}\qquad\phi(X)=\sum_{\{x,y\}\in\binom{X}{2}}\langle{x,y}\rangle.\]
Observe that if $X,Y\subseteq\Gamma$ are disjoint, then
\begin{equation}
\label{eq:phi_union}
  \phi(X\cup Y)=\phi(X)+\phi(Y)+\langle\sigma(X),\sigma(Y)\rangle.
\end{equation}

Let $k\geq2$ be an integer. Let $\boldsymbol{\mathcal{B}}$ be the family of all $B\in\binom\Gamma{k+1}$ satisfying
\begin{equation}
\label{eq:block_conditions}
\sigma(B)=0,\qquad\phi(B)=1,\qquad\text{and}\quad\phi(L)\neq1\quad\text{for every }L\in\binom{B}{k-2}.
\end{equation}
The members of $\boldsymbol{\mathcal{B}}$ are the so-called blocks. We define the $k$-uniform hypergraph $\mathcal{H}_{p,q}$ by letting
\[V(\mathcal{H}_{p,q})=\Gamma\qquad\text{and}\qquad E(\mathcal{H}_{p,q})=\bigcup_{B\in\boldsymbol{\mathcal{B}}}\binom{B}{k}.\]
It is immediate that each block in $\boldsymbol{\mathcal{B}}$ corresponds to a copy of $K^k_{k+1}$ in $\mathcal{H}_{p,q}$.

Below we show that $\mathcal{H}_{p,q}$ is $\cS^kC_4$-free. Note that the exponent-two property of $\Gamma$ plays a crucial role in the proof.

\begin{lemma}
\label{lem:block_C4}
The $k$-uniform hypergraph $\mathcal{H}_{p,q}$ is $S^kC_4$-free.
\end{lemma}

\begin{proof}[Proof of~\Cref{lem:block_C4}]
For any $L\in\binom{\Gamma}{k-2}$, define the link graph $G_L$ as the graph on $\Gamma$, where $\{x,y\}\in E(G_L)$ if and only if $L\cup\{x,y\}\in E(\mathcal{H}_{p,q})$. It suffices to prove that $G_L$ is $C_4$-free for every $L\in\binom{\Gamma}{k-2}$.

Fix an arbitrary $L\in\binom{\Gamma}{k-2}$. Let
\[s=\sigma(L)\qquad\text{and}\qquad d=\phi(L).\]
If $d=1$, the last condition in~\eqref{eq:block_conditions} implies that $G_L$ is empty. Thus, assume without loss of generality that $d\neq1$. For any $\{x,y\}\in E(G_L)$, since $L\cup\{x,y\}\in E(\mathcal{H}_{p,q})$, there is a vertex $z\in\Gamma$ such that
$L\cup\{x,y,z\}\in\boldsymbol{\mathcal{B}}$. Then from~\eqref{eq:block_conditions} it follows that
\[0=\sigma(L\cup\{x,y,z\})=s+x+y+z.\]
Since $\Gamma$ has exponent two, $s=x+y+z$. By~\eqref{eq:phi_union} and~\eqref{eq:block_conditions} we have
\[1=\phi(L\cup\{x,y,z\})=d+\phi(\{x,y,z\})+\langle{s,x+y+z}\rangle=d+\phi(\{x,y,z\}).\]
Moreover,
\begin{equation}
\label{eq:breaking_up}
\phi(\{x,y,z\})=\langle{x,y}\rangle+\langle{x,z}\rangle+\langle{y,z}\rangle=\langle{x+z,y+z}\rangle=\langle{x+s,y+s}\rangle,
\end{equation}
where in the second equality we used $\langle{z,z}\rangle=0$ and in the last equality we used $x+z=y+s$ and $y+z=x+s$. Therefore, every edge $\{x,y\}\in E(G_L)$ satisfies
\[\langle{x+s,y+s}\rangle=1-d\neq0.\]

Suppose that two distinct vertices $x,y\in\Gamma$ have a common neighbor $z$ in $G_L$. Then
\begin{equation}
\label{eq:linear_equations}
\langle{x+s,z+s}\rangle=1-d\qquad\text{and}\qquad\langle{y+s,z+s}\rangle=1-d.
\end{equation}
Since $x\neq y$, we have that $x+s$ and $y+s$ are linearly independent in $\mathbb{F}_q^2$. Indeed, suppose that $x+s=\lambda(y+s)$. Then from~\eqref{eq:linear_equations} it follows that
\[1-d=\langle{x+s,z+s}\rangle=\lambda\langle{y+s,z+s}\rangle=\lambda(1-d).\]
As $1-d\neq0$, this implies $\lambda=1$ and thus $x=y$, a contradiction. Accordingly, there is at most one solution $z\in\mathbb{F}_q^2$ for the linear equation system~\eqref{eq:linear_equations}. Therefore, every pair of vertices has at most one common neighbor in $G_L$, implying that $G_L$ is $C_4$-free. This completes the proof of~\Cref{lem:block_C4}.
\end{proof}

\subsection{Counting blocks}
\label{sec:counting}
We shall repeatedly use the elementary fact (see, e.g., Axler~\cite{A24} for a proof) that if $f:\mathbb{F}_q^d\to\mathbb{F}_q$ is a non-zero linear map, then $|f^{-1}(x)|=q^{d-1}$ for all $x\in\mathbb{F}_q$.

Let $k\geq2$ be an integer and let $2\leq p\leq q$ be powers of two. Let $\mathcal{H}_{p,q}$ be the $k$-uniform hypergraph defined in~\Cref{sec:construction}.

\begin{lemma}
\label{lem:counting_blocks}
$\mathcal{H}_{p,q}$ contains at least
\[\frac{p^kq^{k-1}}{(k+1)!}-p^{k-1}q^{k-1}\]
copies of $K^k_{k+1}$.
\end{lemma}

\begin{proof}[Proof of~\Cref{lem:counting_blocks}]
Since every block in $\boldsymbol{\mathcal{B}}$ corresponds to a unique copy of $K^k_{k+1}$ in $\mathcal{H}_{p,q}$, it suffices to show that
\[\lvert\boldsymbol{\mathcal{B}}\rvert\geq\frac{p^kq^{k-1}}{(k+1)!}-p^{k-1}q^{k-1}.\]
We first count the ordered $(k+1)$-tuples $(x_1,\dots,x_{k+1})\in\Gamma^{k+1}$, such that the multiset $\{x_1,\dots,x_{k+1}\}$ satisfies the first two conditions in~\eqref{eq:block_conditions}. Writing $x_i=(a_i,b_i)\in\Gamma$, the first condition states that
\begin{equation}
\label{eq:first_condition}
\sum_{i=1}^{k+1}a_i=0\qquad\text{and}\qquad\sum_{i=1}^{k+1}b_i=0.
\end{equation}
Since $\Gamma$ has exponent two, after~\eqref{eq:first_condition} is satisfied, the second condition is equivalent to
\begin{equation}
\label{eq:second_condition}
\phi(\{x_1,\dots,x_{k+1}\})=\sum_{i=1}^{k+1}a_ib_i=1.
\end{equation}
Indeed, after expanding the left-hand side of~\eqref{eq:second_condition}, the coefficient of $a_i$ is $\sum_{j\neq i}b_j=0-b_i=b_i$.

There are at least $p^k-p$ ordered tuples $(a_1,\dots,a_{k+1})\in A^{k+1}$ such that $\sum_{i=1}^{k+1}a_i=0$ and $a_i\neq a_j$ for some $i<j$. For any such tuple, we consider the linear map
\[f:\,(b_1,\ldots,b_{k+1})\longmapsto\sum_{i=1}^{k+1}a_ib_i\]
from the $k$-dimensional subspace
\[\left\{(b_1,\dots,b_{k+1})\subseteq\mathbb{F}_q^{k+1}:\,\sum_{i=1}^{k+1}b_i=0\right\}\subseteq\mathbb{F}_q^{k+1}\]
to $\mathbb{F}_q$. Since $a_1,\dots,a_{k+1}$ are not all equal, $f$ is non-zero. Then there are exactly
$q^{k-1}$ solutions for $f(b_1,\dots,b_{k+1})=1$. The number of $(x_1,\dots,x_{k+1})\in\Gamma^{k+1}$ satisfying the first two conditions in~\eqref{eq:block_conditions}, or equivalently~\eqref{eq:first_condition} and~\eqref{eq:second_condition}, is therefore at least
\begin{equation}
\label{good_tuples}
(p^k-p)q^{k-1}=p^kq^{k-1}-pq^{k-1}.
\end{equation}
Moreover, the number of these tuples with repeated entries is at most
\begin{equation}
\label{bad_tuples}
\binom{k+1}{2}|\Gamma|^{k-1}=\binom{k+1}{2}p^{k-1}q^{k-1}.
\end{equation}
Indeed, every such tuple can be counted by choosing a pair of coordinates, specifying their common value, and specifying $k-2$ of the remaining $k-1$ entries. The zero-sum condition~\eqref{eq:first_condition} then uniquely determines the final entry.

By~\eqref{good_tuples} and~\eqref{bad_tuples} there are at least
\[p^kq^{k-1}-(k+1)kp^{k-1}q^{k-1}\]
ordered tuples $(x_1,\dots,x_{k+1})\in\Gamma^{k+1}$ with pairwise distinct entries satisfying the first two conditions in~\eqref{eq:block_conditions}. Dividing by the $(k+1)!$ possible orderings, we obtain at least
\[\frac{p^kq^{k-1}}{(k+1)!}-\frac{p^{k-1}q^{k-1}}{(k-1)!}\]
$B\in\binom{\Gamma}{k+1}$ satisfying the first two conditions in~\eqref{eq:block_conditions}.

It remains to show that the number of $B\in\binom{\Gamma}{k+1}$ satisfying the first two conditions but violating the last one is at most
\[\left(1-\frac{1}{(k-1)!}\right)p^{k-1}q^{k-1}.\]
Since the last condition in~\eqref{eq:block_conditions} is automatically satisfied when $2\leq k\leq3$, we assume that $k\geq4$. We shall first count the number of $L\in\binom{\Gamma}{k-2}$ with $\phi(L)=1$, and then bound the number of $B$ containing a fixed such $L$. We count $L=\{(a_1,b_1),\dots,(a_{k-2},b_{k-2})\}$ by first choosing $a_1,\dots,a_{k-2}$ arbitrarily and then determining $b_1,\dots,b_{k-2}$. After fixing $a_1,\dots,a_{k-2}$, $\phi(L)$ can be considered as a linear map from $\mathbb{F}_q^{k-2}$ to $\mathbb{F}_q$. Consequently, there are at most $q^{k-3}$ solutions $(b_1,\dots,b_{k-2})\in\mathbb{F}_q^{k-2}$ for $\phi(L)=1$. Then the number of $L\in\binom{\Gamma}{k-2}$ with $\phi(L)=1$, after dividing by the $(k-2)!$ possible orderings of $a_1,\dots,a_{k-2}$, is at most
\[\frac{p^{k-2}q^{k-3}}{(k-2)!}.\]
For any $L\in\binom{\Gamma}{k-2}$ with $\phi(L)=1$, we bound the number of $B$ containing $L$ by counting $\{x,y,z\}\subseteq\Gamma\setminus L$ such that $L\cup\{x,y,z\}$ satisfies the first two conditions in~\eqref{eq:block_conditions}. Let $\sigma(L)=s=(s_a,s_b)$ and recall that the group $\Gamma$ has exponent two. The first condition $\sigma(L\cup\{x,y,z\})=0$ yields that
\begin{equation}
\label{eq:first_restriction}
s=x+y+z
\end{equation}
and thus $\langle{s,x+y+z}\rangle=0$. Moreover, the second condition $\phi(L\cup\{x,y,z\})=1$, together with~\eqref{eq:phi_union}, implies that
\begin{equation}
\label{eq:second_restriction}
0=\phi(\{x,y,z\})+\langle{s,x+y+z}\rangle=\phi(\{x,y,z\})=\langle{x+s,y+s}\rangle,
\end{equation}
where the last equality follows from the same transformation used in~\eqref{eq:breaking_up}. Then every such $\{x,y,z\}$ can be counted by first choosing $x=(x_a,x_b)$ arbitrarily and then determining $y=(y_a,y_b)$ and $z$. If $x=s$, then from~\eqref{eq:first_restriction} it holds $y=z$, implying that there are at most $|\Gamma|=pq$ choices for $y$ and $z$. If $x\neq s$, then either $x_a\neq s_a$ or $x_b\neq s_b$. In the case $x_a\neq s_a$, we choose $y_a$ arbitrarily and $y_b$ is then uniquely determined by~\eqref{eq:second_restriction}, and, once $x$ and $y$ are fixed,~\eqref{eq:first_restriction} uniquely determines $z$. The case $x_b\neq s_b$ follows similarly. Accordingly, for each fixed $L$, the number of such $\{x,y,z\}$ is at most
\[pq+(|\Gamma|-1)q\leq\frac{3pq^2}{2}.\]
Therefore, the number of $B\in\binom{\Gamma}{k+1}$ satisfying the first two conditions but violating the last one is at most
\[\frac{p^{k-2}q^{k-3}}{(k-2)!}\cdot\frac{3pq^2}{2}=\frac{3p^{k-1}q^{k-1}}{2(k-2)!}\leq\left(1-\frac{1}{(k-1)!}\right)p^{k-1}q^{k-1},\]
as desired. This completes the proof of~\Cref{lem:counting_blocks}.
\end{proof}

\subsection{Proof of~\Cref{thm:C4}}
\label{sec:proof_thm}
Let $k\geq2$ be an integer and let $n=2^m\geq4$ be a power of two.

Let $\mathcal{H}_{p,q}$ be the $k$-uniform hypergraph defined in~\Cref{sec:construction}, where $p=2^{\lfloor{m/2}\rfloor}$ and $q=2^{\lceil{m/2}\rceil}$. Note that $\mathcal{H}_{p,q}$ has precisely $pq=n$ vertices and by~\Cref{lem:block_C4} is $\cS^kC_4$-free. Moreover, by~\Cref{lem:counting_blocks} the hypergraph $\mathcal{H}_{p,q}$ contains at least
\[\frac{p^kq^{k-1}}{(k+1)!}-p^{k-1}q^{k-1}=\frac{n^kq^{-1}}{(k+1)!}-n^{k-1}\]
copies of $K^k_{k+1}$. Therefore, we have
\[\ex(n,K^k_{k+1},\cS^kC_4)\geq\frac{n^{k-1/2}}{(k+1)!}-n^{k-1}\]
when $m$ is even, in which case $q=\sqrt{n}$, and
\[\ex(n,K^k_{k+1},\cS^kC_4)\geq\frac{n^{k-1/2}}{\sqrt{2}(k+1)!}-n^{k-1}\]
when $m$ is odd, in which case $q=\sqrt{2n}$.\qed

\subsection{Proof of~\Cref{cor:C4}}
\label{sec:proof_cor}
Let $k\geq2$ be a fixed integer. Let $n\in\mathbb{N}$ be sufficiently large. As the upper bound follows readily from~\eqref{upper_bound}, it suffices to establish the lower bound.

Let $m=\lfloor\log_2{n}\rfloor\geq2$ and
\[\lambda_m=\begin{cases}
    1 & \text{if $m$ is even,}\\
    \sqrt{2} & \text{if $m$ is odd.}\end{cases}\]
Let $1\leq x<2$ be such that $n=x2^m$. Then by~\Cref{thm:C4} we have the first lower bound
\begin{equation}
\label{eq:first_lower_bound}
\begin{aligned}
\ex(n,K^k_{k+1},\cS^kC_4)\geq\ex(2^m,K^k_{k+1},\cS^kC_4)&\geq\frac{(2^m)^{k-1/2}}{\lambda_m(k+1)!}-(2^m)^{k-1}\\
&=\frac{1}{\lambda_mx^{k-1/2}}\cdot\frac{n^{k-1/2}}{(k+1)!}-O(n^{k-1}).
\end{aligned}
\end{equation}

For a second lower bound, let $\mathcal{H}$ be an $\cS^kC_4$-free $k$-uniform hypergraph on $2^{m+1}$ vertices attaining $\ex(2^{m+1},K^k_{k+1},\cS^kC_4)$. Let $\mathcal{G}$ be an induced subhypergraph of $\mathcal{H}$ on $n$ vertices chosen uniformly at random. A fixed copy of $K^k_{k+1}$ survives in $\mathcal{G}$ with probability
\[\frac{\binom{2^{m+1}-(k+1)}{n-(k+1)}}{\binom{2^{m+1}}{n}}=\frac{\binom{n}{k+1}}{\binom{2^{m+1}}{k+1}}\geq\left(\frac{x}{2}\right)^{k+1}-O(n^{-1}).\]
Let $X$ be the expected number of copies of $K^k_{k+1}$ in $\mathcal{G}$. Then
\begin{equation}
\label{eq:second_lower_bound}
\begin{aligned}
X&\geq\ex(2^{m+1},K^k_{k+1},\cS^kC_4)\cdot\left(\left(\frac{x}{2}\right)^{k+1}-O(n^{-1})\right)\\
&\geq\frac{(2^{m+1})^{k-1/2}}{\lambda_{m+1}(k+1)!}\cdot\left(\frac{x}{2}\right)^{k+1}-O(n^{k-1})\\
&=\frac{x^{3/2}}{2^{3/2}\lambda_{m+1}}\cdot\frac{n^{k-1/2}}{(k+1)!}-O(n^{k-1}).
\end{aligned}
\end{equation}
Namely, there exists an $n$-vertex induced subhypergraph $\mathcal{G}$ of $\mathcal{H}$ with at least $X$ copies of $K^k_{k+1}$. As $\mathcal{G}$ remains $\cS^kC_4$-free,~\eqref{eq:second_lower_bound} gives a second lower bound on $\ex(n,K^k_{k+1},\cS^kC_4)$. Combining~\eqref{eq:first_lower_bound} and~\eqref{eq:second_lower_bound}, it remains to show that
\[\min\left\{\lambda_mx^{k-1/2},2^{3/2}\lambda_{m+1}x^{-3/2}\right\}\leq2^{2-3/(k+1)}.\]
If $m$ is even, the two terms in the minimum are $x^{k-1/2}$ and $4x^{-3/2}$. Since $x^{k-1/2}\leq2^{2-3/(k+1)}$ when $x\leq2^{2/(k+1)}$ and $4x^{-3/2}\leq2^{2-3/(k+1)}$ when $x\geq2^{2/(k+1)}$, the minimum is at most $2^{2-3/(k+1)}$.
If $m$ is odd, the two terms are $\sqrt{2}x^{k-1/2}$ and $2^{3/2}x^{-3/2}$. Similarly, we have $\sqrt{2}x^{k-1/2}\leq2^{2-3/(k+1)}$ when $x\leq2^{2/(k+1)-1/3}$ and $2^{3/2}x^{-3/2}\leq2^{2-3/(k+1)}$ when $x\geq2^{2/(k+1)-1/3}$. This completes the proof.\qed

\section{Counting cliques in $\cS^kC_\ell$-free hypergraphs}
\label{sec:C_ell}
In this section, we prove~\Cref{prop:longer_cycles}, giving lower and upper bounds on $\ex(n,K^k_t,\cS^kC_\ell)$. The lower bound is obtained using probabilistic methods. Rather than following the standard approach of choosing edges at random and subsequently eliminating bad configurations, our proof draws on the block-packing construction from~\Cref{sec:construction}: we choose blocks at random and then eliminate bad configurations.

\begin{proof}[Proof of~\Cref{prop:longer_cycles}]
Let $k\geq2$ and $\ell\geq3$ be fixed integers. Let $t$ be an arbitrary integer with $k\leq t<k+\ell-2$. We first establish the lower bound on $\ex(n,K^k_t,\cS^kC_\ell)$. Set
\[\alpha=k-1+\frac{1}{\ell-1}\qquad\text{and}\qquad p=\eps n^{\alpha-t}\in(0,1),
\]
where $\eps>0$ is a sufficiently small constant. Let $\mathcal{B}$ be a
random family obtained by including every member of $\binom{[n]}{t}$
independently with probability $p$. Then
\begin{equation}
\label{eq:blocks}
\mathbb{E}(|\mathcal{B}|)=\binom{n}{t}p=\left(\frac{\eps}{t!}+o(1)\right)n^\alpha.
\end{equation}
Define the $k$-uniform hypergraph
$H_{\mathcal{B}}$ by
\[E(H_{\mathcal{B}})=\bigcup_{B\in\mathcal{B}}\binom{B}{k}.\]
Thus, every $B\in\mathcal{B}$ gives a copy of $K^k_t$ in $H_{\mathcal{B}}$.

A bad configuration consists of an ordered tuple of blocks $(B_1,\dots,B_r)\in\mathcal{B}^r$, where $2\leq r\leq\ell$, and an ordered tuple $(E_1,\dots,E_r)$, where $\emptyset\neq E_i\subseteq\binom{B_i}{k}$ for every $i\in[r]$ and $E_1\dot\cup\dots\dot\cup E_r$ gives a copy of $E(\cS^kC_\ell)$. Observe that each copy of $\cS^kC_\ell$ in $H_\mathcal{B}$ is supported by at least one bad configuration. For $2\leq r\leq\ell$, let $X_r$ denote the number of bad configurations $\{(B_1,\dots,B_r),(E_1,\dots,E_r)\}$.

We shall bound $X_r$ by first choosing $(E_1,\dots,E_r)$ and then choosing $(B_1,\dots,B_r)$. As $\cS^kC_\ell$ has $k+\ell-2$ vertices, there are at most $n^{k+\ell-2}$ copies $F$ of $\cS^kC_\ell$ in $H_\mathcal{B}$. Since the number of ordered partitions of $E(F)$ into $r$ parts is upper bounded by a constant, there are $O(n^{k+\ell-2})$ choices of $(E_1,\dots,E_r)$. Now, assume that $(E_1,\dots,E_r)$ is fixed and we bound the number of $(B_1,\dots,B_r)$ such that $\{(B_1,\dots,B_r),(E_1,\dots,E_r)\}$ forms a bad configuration. For every $i\in[r]$, since $\emptyset\neq E_i\in\binom{B_i}{k}$, at least $k+|E_i|-1$ vertices of $B_i$ are fixed. Consequently, there are at most $n^{t-(k+|E_i|-1)}$ choices for $B_i$. The number of $(B_1,\dots,B_r)$ is then at most
\[\prod_{i=1}^{r}n^{t-(k+|E_i|-1)}=n^{r(t-k+1)-\sum_{i=1}^{r}|E_r|}=n^{r(t-k+1)-\ell},\]
which yields that $X_r=O(n^{k-2+r(t-k+1)})$. Let $X=\sum_{r=2}^{\ell}X_r$ be the number of all bad configurations. Then
\begin{equation}
\label{eq:X}
\mathbb{E}(X)=\sum_{r=2}^\ell\mathbb{E}(X_r)=\sum_{r=2}^{\ell}O(n^{k-2+r(t-k+1)})p^r=\sum_{r=2}^{\ell}O(\eps^r n^{k-2+r/(\ell-1)})=O(\eps^\ell n^\alpha).
\end{equation}

Since $\eps>0$ is a sufficiently small constant, combining~\eqref{eq:blocks} and~\eqref{eq:X} we have
\[\mathbb{E}(|\mathcal{B}|-X)=\mathbb{E}(|\mathcal{B}|)-\mathbb{E}(X)=\Omega(n^\alpha).\]
Then there exists some $H_\mathcal{B}$ for which $|\mathcal{B}|-X=\Omega(n^\alpha)$. Delete a block from each bad configuration in $H_\mathcal{B}$ and let $H$ be the resulting hypergraph. Then $H$ is $\cS^kC_\ell$-free and contains at least $|\mathcal{B}|-X=\Omega(n^\alpha)$ copies of $K^k_t$. This yields the lower bound
\[\ex(n,K^k_t,\cS^kC_\ell)=\Omega(n^{k-1+1/(\ell-1)}).\]

For the upper bounds, we first claim that, for every $2\leq r<\ell$, a $C_\ell$-free graph $G$ contains at most $(\ell-3)^{r-2}|E(G)|$ copies of $K_r$. Indeed, when $r=2$ this is trivial. Assume that $r>2$. For any vertex $v\in V(G)$, let $G_v$ denote the subgraph of $G$ induced by the neighborhood of $v$. Since $G$ is $C_\ell$-free, $G_v$ contains no path on $\ell-1$ vertices. Then $G_v$ must be $(\ell-3)$-degenerate. Equivalently, the vertices of $G_v$ can be written in an order, such that every vertex has at most $\ell-3$ neighbors among its preceding vertices. It then follows that $G_v$ contains at most $|V(G_v)|(\ell-3)^{r-2}$ copies of $K_{r-1}$. As each copy of $K_r$ in $G$ corresponds a copy of $K_{r-1}$ in exactly $r$ distinct $G_v$, there are at most
\[\frac{\sum_{v\in V(G)}|V(G_v)|(\ell-3)^{r-2}}{r}\leq(\ell-3)^{r-2}|E(G)|\]
copies of $K_r$ in $G$.

Now let $H$ be an $n$-vertex $\cS^kC_\ell$-free $k$-uniform hypergraph. For every $L\in\binom{V(H)}{k-2}$, let $G_L$ be the link graph on $V(H)$, where $\{x,y\}\in E(G_L)$ if and only if
$L\cup\{x,y\}\in E(H)$. Then, every
$G_L$ is $C_\ell$-free, and every copy of $K^k_t$ in $H$ corresponds to a copy of $K_{t-k+2}$ in exactly $\binom{t}{k-2}$ distinct $G_L$. Recall that from our claim, $G_L$ contains at most $(\ell-3)^{t-k}|E(G_L)|$ copies of $K_{t-k+2}$. Moreover, as $G_L$ is $C_\ell$-free, from~\eqref{C6C10} we have $|E(G_L)|=O(n^2)$ if $\ell$ is odd and $|E(G)|=O(n^{1+2/\ell})$ if $\ell$ is even. Consequently, $H$ contains at most
\[\frac{\sum_{L\in\binom{V(H)}{k-2}}(\ell-3)^{t-k}|E(G_L)|}{\binom{t}{k-2}}=\begin{cases}
O(n^k) & \text{if $\ell$ is odd,}\\
O(n^{k-1+2/\ell}) & \text{if $\ell$ is even.}
\end{cases}\]
copies of $K^k_t$. This gives the desired upper bound on $\ex(n,K^k_t,\cS^kC_\ell)$, completing the proof.
\end{proof}

\section{Tur\'an numbers of $\cS^kC_6$ and $\cS^kC_{10}$}
\label{sec:C6C10}
In this section we prove~\Cref{thm:C6C10} by developing $k$-uniform analogues of Wenger's constructions.

\subsection{Uniform lifts of Wenger's constructions}
Given an integer $\ell\geq2$ and a prime power $q$, the Wenger graph $W_{\ell,q}$ is the bipartite graph whose two vertex classes are copies of $\mathbb{F}_q^\ell$. For two vertices $u=(u_1,\dots,u_\ell)$ and $v=(v_1,\dots,v_\ell)$ from different classes, join $u$ to $v$ in $W_{\ell,q}$ if and only if
\[u_d+v_d=u_1v_1^{d-1}\qquad\text{for all $d=2,\dots,\ell$.}\]
Wenger~\cite{W91} proved that $W_{\ell,q}$ is $C_{2\ell}$-free for $\ell\in\{2,3,5\}$. Below we introduce the $k$-uniform version of Wenger graphs.

\begin{definition}
Let $k,\ell\geq2$ be integers and let $q$ be a prime power. For each $i\in[k]$ let $V_i=\{i\}\times\mathbb{F}_q^\ell$.
Define $W^k_{\ell,q}$ to be the $k$-uniform hypergraph on
$V_1\dot\cup\dots\dot\cup V_k$, where, writing $x_i=(x_{i,1},\dots,x_{i,\ell})\in\mathbb{F}_q^\ell$, 
\[\{(1,x_1),\dots,(k,x_k)\}\in E(W^k_{\ell,q})\]
if and only if
\begin{equation}
\label{eq:k_uniform_Wenger}
\sum_{i=1}^k x_{i,d}=\sum_{1\leq i<j\leq k}x_{i,1}x_{j,1}^{d-1}\qquad\text{for all $d=2,\dots,\ell$.}
\end{equation}
\end{definition}

Note that $W^k_{\ell,q}$ is $k$-partite by definition, and $W^2_{\ell,q}=W_{\ell,q}$. Shortly after completing this work, we learned that $W^k_{\ell,q}$ also belongs to the family of algebraically defined hypergraphs introduced by Lazebnik and Mubayi~\cite{LM02}. The number of edges in these hypergraphs was determined in~\cite{LM02}. For completeness, we give the corresponding edge count for $W^k_{\ell,q}$ here.

\begin{lemma}
\label{lem:Wenger_edges}
$|E(W^k_{\ell,q})|=q^{k\ell-\ell+1}$.
\end{lemma}

\begin{proof}[Proof of~\Cref{lem:Wenger_edges}]
We shall count the number of $\{(1,x_1),\dots,(k,x_k)\}$ satisfying~\eqref{eq:k_uniform_Wenger}. Let us start by choosing the first coordinates $x_{1,1},\dots,x_{k,1}$ arbitrarily. After the first coordinates are fixed, for each $d\in\{2,\dots,\ell\}$,~\eqref{eq:k_uniform_Wenger} is a linear equation in
$x_{1,d},\ldots,x_{k,d}$ and has exactly $q^{k-1}$ solutions. Thus,
\[|E(W^k_{\ell,q})|=q^k\cdot(q^{k-1})^{\ell-1}=q^{k\ell-\ell+1}\]
as claimed.
\end{proof}

Next we prove our key lemma, asserting that $W^k_{\ell,q}$ is $\cS^kC_{2\ell}$-free. This statement does not follow from previous work and constitutes the main difficulty in lifting Wenger graphs to the $k$-uniform setting. We resolve this issue using a coordinate-changing trick.

\begin{lemma}
\label{lem:Wenger_C6C10}
$W^k_{\ell,q}$ is $\cS^kC_{2\ell}$-free for $\ell\in\{2,3,5\}$.
\end{lemma}

\begin{proof}[Proof of~\Cref{lem:Wenger_C6C10}]
For any $L\in\binom{V(W^k_{\ell,q})}{k-2}$, let $G_L$ denote the link graph on $V(W^k_{\ell,q})$, where $\{u,v\}\in E(G_L)$ if and only if $L\cup\{u,v\}\in E(W^k_{\ell,q})$. To show that $W^k_{\ell,q}$ contains no copy of $\cS^kC_{2\ell}$, it suffices to prove that $G_L$ is $C_{2\ell}$-free for every $L\in\binom{V(W^k_{\ell,q})}{k-2}$.

Fix an arbitrary $L\in\binom{V(W^k_{\ell,q})}{k-2}$. Assume without loss of generality that $L$ intersects $k-2$ distinct vertex classes of $W^k_{\ell,q}$, as otherwise $G_L$ is empty and our assertion follows. Let $V_a$ and $V_b$ be the two vertex classes disjoint from $L$, where $a<b$. Write
\[L=\{(i,x_i):\,i\in[k]\setminus\{a,b\}\}\qquad\text{and}\qquad x_i=(x_{i,1},\dots,x_{i,\ell})\in\mathbb{F}_q^\ell.\]
Let $G_L'$ be the induced subgraph of $G_L$ on $V_a\cup V_b$. As all non-isolated vertices of $G_L$ lie in $V_a\cup V_b$, it suffices to show that $G_L'$ is $C_{2\ell}$-free.

Let $(a,u)\in V_a$ and $(b,v)\in V_b$, where
\[u=(u_1,\ldots,u_\ell)\in\mathbb{F}_q^\ell\qquad\text{and}\qquad v=(v_1,\ldots,v_\ell)\in\mathbb{F}_q^\ell.\]
For $d=2,\dots,\ell$, define
\begin{align*}
P_d(Z)&=\sum_{i\in[k]\setminus\{a,b\},i<a}x_{i,1}Z^{d-1}+\sum_{j\in[k]\setminus\{a,b\},j>a}Zx_{j,1}^{d-1},\\
Q_d(Z)&=\sum_{i\in[k]\setminus\{a,b\},i<b}x_{i,1}Z^{d-1}+\sum_{j\in[k]\setminus\{a,b\},j>b}Zx_{j,1}^{d-1},\\
c_d&=\sum_{i,j\in[k]\setminus\{a,b\},i<j}x_{i,1}x_{j,1}^{d-1}-\sum_{i\in[k]\setminus\{a,b\}}x_{i,d}.
\end{align*}
Separating in~\eqref{eq:k_uniform_Wenger} the terms involving $u$, the
terms involving $v$, and the terms determined by $L$, we see that
$L\cup\{(a,u),(b,v)\}$ is an edge of $W^k_{\ell,q}$ if and only if
\begin{equation}
\label{eq:link_Wenger}
u_d+v_d=u_1v_1^{d-1}+P_d(u_1)+Q_d(v_1)+c_d\qquad\text{for all $d=2,\dots,\ell$.}
\end{equation}

Define the coordinate maps $\Phi,\Psi:\mathbb{F}_q^\ell\to\mathbb{F}_q^\ell$ by
\begin{align*}
\Phi(u)&=\left(u_1,u_2-P_2(u_1),\dots,u_\ell-P_\ell(u_1)\right),\\
\Psi(v)&=\left(v_1,v_2-Q_2(v_1)-c_2,\dots,v_\ell-Q_\ell(v_1)-c_\ell\right).
\end{align*}
It is immediate that both maps are bijective. Moreover, under the coordinate changes, the system of equations~\eqref{eq:link_Wenger} transforms into
\begin{equation}
\label{eq:new_coordinates}
\Phi(u)_d+\Psi(v)_d=\Phi(u)_1\Psi(v)_1^{d-1}\qquad\text{for all $d=2,\dots,\ell$.}
\end{equation}
Now since two vertices $(a,u)$ and $(b,v)$ are adjacent in $G_L'$ if and only if~\eqref{eq:new_coordinates} holds, the graph $G_L'$ is isomorphic to the original Wenger graph $W_{\ell,q}$. Recall that $W_{\ell,q}$ is $C_{2\ell}$-free for $\ell\in\{2,3,5\}$. This completes the proof of~\Cref{lem:Wenger_C6C10}.
\end{proof}

\subsection{Proof of~\Cref{thm:C6C10}}
Let $k\geq2$ be a fixed integer and $\ell\in\{3,5\}$. Let $n\in\mathbb{N}$ be sufficiently large. The upper bound $\ex(n,\cS^kC_{2\ell})=O(n^{k-1+1/\ell})$ follows from~\Cref{prop:longer_cycles}. To establish the lower bound, we consider the $k$-uniform hypergraph $W^k_{\ell,q}$, where $(1-o(1))(n/k)^{1/\ell}\leq q\leq(n/k)^{1/\ell}$ is a prime power. In particular, $\lvert{V(W^k_{\ell,q})}\rvert=kq^\ell\leq n$. Since $W^k_{\ell,q}$ is $\cS^kC_{2\ell}$-free by~\Cref{lem:Wenger_C6C10},
\[\ex(n,\cS^kC_{2\ell})\geq\lvert{E(W^k_{\ell,q})}\rvert=q^{k\ell-\ell+1}=\Omega(n^{k-1+1/\ell}),\]
where we used the edge count from~\Cref{lem:Wenger_edges} together with $q\geq(1-o(1))(n/k)^{1/\ell}$.\qed

\section{Applications to simplicial Tur\'an problems}
\label{sec:applications}
Given an integer $k\geq2$, a $(k-1)$-dimensional (abstract) simplicial complex $\FF=(V(\FF),E(\FF))$ is a non-uniform hypergraph without isolated vertices, where $E(\FF)$ is closed under taking subsets and has maximum edge size $k$. Let $\ex(n,\FF)$ denote the maximum number of edges in an $n$-vertex $\FF$-free simplicial complex. The systematic study of simplicial Tur\'an problems was recently initiated by Conlon, Piga, and Sch\"ulke~\cite{CPS23} and subsequently developed by Axenovich, Gerbner, Patk\'os, and the author~\cite{AGLP25}.

For a $k$-uniform hypergraph $H$ without isolated vertices, let $\cD(H)$ denote the $(k-1)$-dimensional simplicial complex obtained by taking the downward closure of $E(H)$. It was shown in~\cite{AGLP25} that
\begin{equation}
\label{eq:simplicial}
\ex(n,\cD(H))=\ex^{\cl}(n,H)+\sum_{r=0}^{k-1}\binom{n}{r},
\end{equation}
where $\ex^{\cl}(n,H)$ is the maximum number of cliques of order at least $k$ contained in an $n$-vertex $H$-free $k$-uniform hypergraph. In~\cite{AGLP25} it was observed that $\ex(n,\cD(\cS^kC_4))=\Theta(n^{k-1/2})$. As a direct application of~\Cref{thm:C4}, we determine the asymptotics of $\ex(n,\cD(\cS^kC_4))$ for infinitely many $n$.

\begin{corollary}
\label{cor:simplicial_C4}
Let $k\geq2$ be an integer. Let $n$ range over the sequence of powers of four. Then
\[\ex(n,\cD(\cS^kC_4))=\left(\frac{1}{k!}+\frac{1}{(k+1)!}\right)n^{k-1/2}+O(n^{k-1}).\]
\end{corollary}

\begin{proof}[Proof of~\Cref{cor:simplicial_C4}]
Recall that any $\cS^kC_4$-free $k$-uniform hypergraph contains no clique of order greater than $k+1$. Then, from~\eqref{eq:simplicial} it follows that
\begin{align*}
\ex(n,\cD(\cS^kC_4))=\ex^{\cl}(n,\cS^kC_4)+O(n^{k-1})&\leq\ex(n,\cS^kC_4)+\ex(n,K^k_{k+1},\cS^kC_4)+O(n^{k-1})\\
&\leq\left(\frac{1}{k!}+\frac{1}{(k+1)!}\right)n^{k-1/2}+O(n^{k-1}),
\end{align*}
where in the last step we used~\eqref{Mubayi_C4} and~\eqref{upper_bound}. For the lower bound, note that every edge in an $\cS^kC_4$-free hypergraph is contained in at most one copy of $K^k_{k+1}$. Thus, if an $\cS^kC_4$-free $k$-uniform hypergraph contains $\ex(n,K^k_{k+1},\cS^kC_4)$ copies of $K^k_{k+1}$, then it necessarily has at least $(k+1)\ex(n,K^k_{k+1},\cS^kC_4)$ edges. Then by~\Cref{thm:C4} we have
\begin{align*}
\ex(n,\cD(\cS^kC_4))&=\ex^{\cl}(n,\cS^kC_4)+O(n^{k-1})\\
&\geq(k+1)\ex(n,K^k_{k+1},\cS^kC_4)+\ex(n,K^k_{k+1},\cS^kC_4)+O(n^{k-1})\\
&\geq\left(\frac{1}{k!}+\frac{1}{(k+1)!}\right)n^{k-1/2}+O(n^{k-1}).\qedhere
\end{align*}
\end{proof}

Furthermore, it was asked in~\cite{AGLP25} which non-integer exponents other than $k-1/2$ can occur for the extremal numbers of $(k-1)$-dimensional simplicial complexes. As an application of~\Cref{thm:C6C10}, we prove that, for every $k\geq2$, there exist $(k-1)$-dimensional simplicial complexes whose extremal numbers have exponents $k-2/3$ and $k-4/5$. This yields two new possible exponents for $\ex(n,\FF)$.

\begin{corollary}
\label{cor:simplicial_C6C10}
Let $k\geq2$ be an integer and $\ell\in\{3,5\}$. Then
\[\ex(n,\cD(\cS^kC_{2\ell}))=\Theta(n^{k-1+1/\ell}).\]
\end{corollary}

\begin{proof}[Proof of~\Cref{cor:simplicial_C6C10}]
For the lower bound, by~\Cref{thm:C6C10} and~\eqref{eq:simplicial} we have
\[\ex(n,\cD(\cS^kC_{2\ell}))\geq\ex^{\cl}(n,\cS^kC_{2\ell})\geq\ex(n,\cS^kC_{2\ell})=\Theta(n^{k-1+1/\ell}).\]
For the upper bound, we apply~\Cref{prop:longer_cycles} together with~\eqref{eq:simplicial}, which gives
\begin{align*}
\ex(n,\cD(\cS^kC_{2\ell}))\leq\ex^{\cl}(n,\cS^kC_{2\ell})+O(n^{k-1})&\leq\sum_{t=k}^{k+2\ell-2}\ex(n,K^k_t,\cS^kC_{2\ell})+O(n^{k-1})\\
&=O(n^{k-1+1/\ell}).\qedhere
\end{align*}
\end{proof}

\section*{Acknowledgements}
The author would like to thank D\'aniel Gerbner for bringing Mubayi's results on $\ex(n,\cS^kK_{s,t})$ to his attention, and Maria Axenovich for helpful discussions and for pointing him to the reference~\cite{LUW95}. The author declares the use of ChatGPT 5.6 Pro to discover the algebraic constraints in~\eqref{eq:block_conditions} for the case $\Gamma=\mathbb{F}_q^2$, as well as the initial coordinate-changing trick in~\Cref{lem:Wenger_C6C10} for $k=\ell=3$. After using this tool, the author reviewed, generalized, and independently completed the proof. The author takes full responsibility for the content of the article.

\end{document}